\documentclass[11pt]{amsart}

\usepackage{amsthm,amssymb,graphicx,amsmath,pb-diagram,xcolor,colortbl,hyperref}
\usepackage{amsfonts}
\usepackage{verbatim}
\usepackage[T1]{fontenc}
\usepackage{graphicx}

\newtheorem{theorem}{Theorem}

\newtheorem{lemma}{Lemma}

\begin{document}

\author{Robbert Fokkink and  Dan Rust}


\title{A modification of Wythoff Nim}

\begin{abstract}
We modify Wythoff's game by allowing an additional move,
which we call a \textit{split},
and show how
the $P$-positions
are coded by the Tribonacci word.
We analyze the table of letter positions of
arbitrary $k$-bonacci words and find
a $\mathrm{mex}$-rule that generates the
Quadribonacci table.
\end{abstract}

\maketitle

In an impartial game, the set of all positions is divided into
those that are winning and those that are losing for the player
that moves next. To play a game optimally, one needs to know
the subset $\mathcal P$ of all losing positions. 
The hardness of solving
an impartial game therefore depends on the complexity of~$\mathcal P$,
and so one runs into set theory.
Substitution sets are easy to describe
and have a well-developed theory~\cite{L}.
Eric Duch\^{e}ne, Michel Rigo, and others, are building a
framework that connects impartial games to substitutions,
~\cite{ DMPR, DR, DR3, DR2}.
A basic example is the
Fibonacci substitution $0\mapsto 01$, $1\mapsto 0$ which leads
to the infinite sequence $010010100100101\cdots$. 
If we take this to be the indicator function of a subset of $\mathbb N$,
then we find the losing positions $\mathcal P$ of Wythoff Nim.
Duch\^ene and Rigo found other impartial games that can be described by
other substitutions~\cite{DR,DR3}.
They asked in~\cite{DR} if it is possible to devise an impartial game that can
be described by the $k$-bonacci substitution.
This was the motivating question for our paper.

We were able to find a simple extension of Wythoff Nim that can be described
by the $3$-bonacci substitution. We were also able to find a $\mathrm{mex}$-relation
to generate the table of the $4$-bonacci substitution, but we were not able to
derive an impartial game from this. Our game, which we call \emph{Splythoff Nim},
is closely related to the Greedy Queen on a Spiral Game, which was recently solved
by Dekking, Shallit and Sloane~\cite{DSS}. Splythoff Nim and Greedy Queen
have the same $P$-positions but different Sprague-Grundy values.

Our paper is organized as follows. We first recall Wythoff Nim and the games and
number tables that arose out of it. Then we introduce Splythoff Nim, which
we solve by an analysis of number tables that come out of the $k$-bonacci
substitution. Finally, we present some numerical evidence that indicates that
$a$-Splythoff, which is the corresponding modification of $a$-Wythoff, can
be coded by a substitution if $a=2$ or $3$.

\section{Wythoff Nim}

The following take-away game was invented by Willem Wythoff
in \cite{W} and is now known as Wythoff's game or Wythoff Nim:
\emph{Two piles of counters are placed on the table.
Two players alternately either take an arbitrary number of counters from a single pile
or an equal arbitrary number from both piles. The player who takes the last counter,
or counters, wins.} It is one of the first impartial games to be reported
in the mathematical
literature, second only to Bouton's analysis of Nim.

In an impartial game, a position is either winning or losing for the
player that moves next. The winning positions are called the $N$-positions and
the losing positions are called the $P$-positions.
The first few non-zero $P$-positions of Wythoff Nim are contained in the table below.
There are several methods to generate this table.
Two methods were already presented by Wythoff. His first method was to generate the
table recursively, starting with the smallest $P$-position $(0,0)$
(which we don't list in the table). To find each next $P$-position,
list the smallest number that does not yet occur in the $A$ row and add $k+1$
to it for the $B$-row, where $k$ is the difference between $A$ and $B$ in the previous position.
His second method gave an explicit equation $A_k=\lfloor k\phi\rfloor$,
$B_k=\lfloor k\phi^2\rfloor$, where $\phi$ is the golden mean.
Twenty years after Wythoff's paper,
such complementary sequences became known as \emph{Beatty sequences}.
\begin{table}[htbp]
\begin{tabular}{c|cccccccccccccccc}
\hline
$A$&1&3&4&6&8&9&11&12&14&16&17&19&$\cdots$ \\
$B$&2&5&7&10&13&15&18&20&23&26&28&31&$\cdots$\\ \hline
\end{tabular}
\caption{\small{The first twelve non-zero $P$-positions for Wythoff Nim. The number of counters in
the smaller pile in $A$. The larger pile in $B$.}\label{tbl1}}
\end{table}

Wythoff Nim continues to be of interest
today and many modifications of
the game have been studied, see~\cite{DFGHKL} for a comprehensive overview.
Wythoff himself already proposed further modifications
of the game, although it seems
that this has been overlooked. 
Wythoff was not a man of many words~\cite{Fok} and
he only mentioned these modifications very briefly
in two final remarks.
These remarks were overlooked and only
rediscovered much later.
Wythoff's first remark was that for any $a$ it is possible to define complementary
sequences $A_k=\lfloor k\alpha\rfloor$ and
$B_k=\lfloor k\beta\rfloor$ for certain quadratic numbers
$\alpha$ and $\beta$ such that $B_k-A_k=ka$ for $k=1,2,3,\ldots$.
Wythoff only gave the $\alpha$ and $\beta$, and 
left the games that belong to these sequences
to the reader. It is difficult to guess what rules
he had in mind.
These games were rediscovered by
Holladay~\cite{H}, who gave four different take-away games that all lead to these 
$P$-positions.
The most natural rule is to allow the player to take an arbitrary number
of counters from both piles, as long as the difference between these
numbers is less than $a$, as also considered in~\cite{Front}.
This game is now called \emph{$a$-Wythoff}. The original game corresponds
to~$a=1$.
\begin{table}[htbp]
\begin{tabular}{c|ccccccccccccccccc}
\hline
$A$& 1& 2& 4& 5& 7& 8& 9&11&12&14&15&16&$\cdots$\\
$B$& 3& 6&10&13&17&20&23&27&30&34&37&40&$\cdots$\\
\hline
\end{tabular}
\caption{\small{The first twelve $P$-positions for $2$-Wythoff.
This table also appears in another generalization of
Wythoff Nim, see~\cite[p. 188]{C}.}\label{tbl2}}
\end{table}

One can study such tables in their own right. Wythoff's second remark
was that one can define
complementary
sequences $A_k=\lfloor k\alpha+\gamma\rfloor$ and $B_k=\lfloor k\beta+\delta\rfloor$
for certain quadratic numbers $\alpha,\beta,\gamma,\delta$ such
that $B_k-A_k=ka+b$ for arbitrary $0< b\leq a$ and $k=0,1,2,\ldots$.
If $b=a$ we recover the sequences for $a$-Wythoff.
Kimberling~\cite{K} rediscovered $A_k$ and
$B_k$ not so long ago, and called them $ka\!+\!b$ Wythoff sequences.
They
are special examples of \emph{non-homogeneous Beatty sequences}.
Skolem~\cite{Sk} and Fraenkel~\cite{Fnonhom} found
necessary and sufficient conditions on
$\alpha,\beta,\gamma,\delta$
such that the sequences are complementary:
\begin{eqnarray}
\frac 1\alpha+\frac 1\beta&=&1\\
\frac \gamma\alpha+\frac \delta\beta&=&\lfloor\alpha+\gamma\rfloor.
\end{eqnarray}
Wythoff left no clue as to how he found his $\alpha,\beta,\gamma,\delta$.
With the benefit of hindsight the numbers can easily
be computed from the Skolem-Fraenkel conditions
combined with $\beta\!-\!\alpha\!=\!a$ and $\delta\!-\!\gamma\!=\!b$.
These $(ka\!+\!b)$-sequences do not correspond to $P$-positions
in a readily defined take-away game.
For instance, $(1,2)$ is not in Table~\ref{tbl3} for the
$k\!+\!1$-Wythoff sequences,
and the only losing position within reach is $(0,0)$.
Now both $(2,5)$ to $(1,3)$ are in the table,
and moves between $P$-positions are impossible by definition.
If Table~\ref{tbl3} would represent $P$-positions, then
taking away $(1,2)$ would be forbidden from position
$(2,5)$ but allowed from position $(1,2)$.
A take-away game is called \emph{invariant} if forbidden moves do not depend on
positions, and so we see that there is no invariant game with $P$-positions
as in Table~\ref{tbl3}.
Larsson, Hegarty, and Fraenkel showed that for every pair
of complementary Beatty
sequences, there exists an invariant game which has
$P$-positions along these sequences.
This work has been extended to non-homogeneous Beatty sequences
in~\cite{CDR}.
\begin{table}[htbp]
\begin{tabular}{c|ccccccccccccccccc}
\hline
$A$& 1& 2& 4& 6& 7& 9&10&12&14&15&17&19&$\cdots$\\
$B$& 3& 5& 8&11&13&16&18&21&24&26&29&32&$\cdots$\\
\hline
\end{tabular}
\caption{\small{The first twelve entries for Wythoff's
sequences $B_k-A_k=k+1$.
}\label{tbl3}}
\end{table}

The tables that we encountered so far can be neatly described
by \textit{substitutions}. The Fibonacci substitution is
\[
\begin{array}{ccl}
0&\mapsto& 01\\
1&\mapsto& 0
\end{array}.
\]
If we iterate this substitution starting from $0\mapsto 01\mapsto 010\mapsto 01001\mapsto\cdots$
then in the limit we get the \emph{Fibonacci word} $010010100100101\cdots$, which is fixed
under the substitution.
It is well known that the numbers $A_k$ and $B_k$ in Table 1 correspond to
the locations of the zeroes and ones in this word, see~\cite[p. 381]{DR}.
Duch\^{e}ne and Rigo~\cite{DR3} proved that the Beatty sequences for $a$-Wythoff
correspond to the locations of the zeroes and ones in the fixed word
of $0\mapsto 0^a1,\ 1\mapsto 0$, where $0^a$ denotes the word with $a$ zeroes.
This is known as a \emph{noble means} substitution~\cite[p. 91]{BG}.
The numbers $A_k$ and $B_k$ in Table 2 correspond to the locations of the zeroes
and ones in the \textit{Pell word}, which arises from the substitution
$0\mapsto 001,\ 1\mapsto 0$, and is listed as 
\href{https://oeis.org/A171588}{A171588} in the On-Line
Encyclopedia of Integer Sequences.
It is possible to prove that the $ka\!+\!b$ Wythoff sequences
correspond to locations in the
fixed word of $0\mapsto 0^b10^{a-b},\ 1\mapsto 0$.
All these fixed words and substitutions are \emph{Sturmian}, see~\cite[Chapter 2]{Lot}.
Morse and Hedlund~\cite{MH} proved that all Sturmian words are of the form $U(\rho,x)=u_1u_2u_3\cdots$
with $u_{k}=\lfloor k\rho + x\rfloor-\lfloor (k-1)\rho +x\rfloor$
for some irrational $\rho\in(0,1)$ and some $x\in [0,1)$
(possibly interchanging $0$ and $1$ in the sequence).
The zeroes in $U(\rho,x)$ are at the locations $\lfloor k(\rho+1)+x\rfloor$.
To prove that the substitution $0\mapsto 0^a10^{b-a},\ 1\mapsto 0$
gives the $ak\!+\!b$ Wythoff sequence, one needs to convert the Sturmian word
to $U(\rho,x)$ and verify that $\rho+1=\alpha$ and $x=\gamma$
for Wythoff's $\alpha,\gamma$.
It is possible to do this using an algorithm of
Arnoux, et al~\cite{AFH}.
Conversely, it is possible to compute the Sturmian word from $U(\rho,x)$
using an algorithm of Ito and Yasutomi~\cite{IY}.
Both algorithms involve continued fraction expansions.
It turns out that the continued fraction expansion of $\alpha$
in the $ak\!+\!b$ Wythoff sequence is $[1;a,a,a,\ldots]$, which goes back to A.A. Markov, see
\cite[Theorem 1]{IY}.

\section{Splythoff Nim}

We modify Wythoff Nim by allowing the additional option of a \emph{split}. If
a player takes an equal number of counters from both piles and only one pile
remains, then he can split the remaining pile into two. For instance,
from position $(4,7)$ it is possible to move to $(1,2)$ by taking 4 counters from both piles and
splitting the remainder into 1 and 2. A split is only allowed after taking counters from both
piles. It is not allowed to take all counters from a single pile and then split.
We call this \emph{Splythoff Nim}. We refer to the three possible moves as single,
double, and split. The first few non-zero $P$-positions are in Table~\ref{tbl4} below.

\begin{table}[htbp]
\begin{tabular}{c|cccccccccccccccc}
\rowcolor[gray]{.8}$\Delta$&
1&2&4&5&6&7&9&10&11&13&14&15&$\cdots$ \\
\hline
$A$&1&3&4&6&7&9&10&12&14&15&17&18&$\cdots$ \\
$B$&2&5&8&11&13&16&19&22&25&28&31&33&$\cdots$ \\
\hline
\rowcolor[gray]{.8}$\Sigma$&3&8&12&17&20&25&29&34&39&43&48&51&$\cdots$
\end{tabular}
\caption{\small{The first twelve non-zero $P$-positions for Splythoff Nim, alongside the sums
and differences of these positions.
}\label{tbl4}}
\end{table}

The option of splitting a pile places Splythoff Nim in the class of take-and-break games
such as Lasker's Nim~\cite{BCG}.
A version of Wythoff Nim that is somewhat similar to our game
involves 'splitting pairs', see~\cite{L}. This notion of splitting is
different from ours.	
Splythoff Nim is not an invariant game since forbidden moves depend
on positions. For instance,
it is possible to move from $(4,7)$ to $(1,2)$ but it is not possible to move from
$(3,5)$ to $(0,0)$.

If we code the $A$ and $B$ sequences by an infinite word of zeroes and ones then we get
\begin{equation}\label{ABword}
01001001001010010010010010010010100100100100100101001001\cdots.
\end{equation}
It turns out that this sequence cannot be produced from a substitution on two letters,
because the frequency of zeroes in the sequence is not a quadratic number. It is a cubic,
as we shall see later. However, it is possible to \emph{code} the infinite word by the
Tribonacci substitution $0\mapsto 01,\ 1\mapsto 02,\ 2\mapsto 0$ which gives the Tribonacci word
\begin{equation}\label{codedword}
010201001020101020100102010201001020101020100102010010201	\cdots.
\end{equation}
If we recode this sequence by deleting all $2$'s then we get the infinite word corresponding
to $A$ and $B$ in the table. This follows from Theorem~\ref{codingnu} below,
combined with our main result:

\begin{theorem}\label{thm1}
Let  $x_i, y_i, z_i$ be the locations of the $i$-th $0,1,2$, respectively,
in the Tribonacci word.
Then the $i$-th $P$-position in Splythoff's Nim $(a_i,b_i)$ is given by
$a_i=y_i-x_i$ and $b_i=z_i-y_i$.
\end{theorem}

A word on notation. From now on, we reserve capital letters for sets or sequences and small letters for
numbers. For instance, we write $A$ for the sequence of first coordinates of $P$-positions $a_i$ in
Splythoff Nim. The theorem says that $A=Y-X$ and $B=Z-Y$ if $X,Y,Z$ are the locations of the letters in
the Tribonacci word. In particular, the Tribonacci word has the remarkable property that the differences
$Y-X$ and $Z-Y$ are complementary, as observed by Duch\^ene and Rigo~\cite[Cor. 3.6]{DR}. 
It turns out that this is a property of all $k$-bonacci words.

Theorem~\ref{thm1} is an analogy of the description of Table~\ref{tbl1} as locations in the Fibonacci word.
We already saw Wythoff gave two other descriptions of this table.
His first description generated the columns recursively. In the next section, we will do that for Table~\ref{tbl4}.

$A$ is sequence \href{https://oeis.org/A140100}{A140100} and $B$ is sequence \href{https://oeis.org/A140101}{A140101} in the OLEIS. The sequences
arise from the Greedy Queens in a spiral problem, which has recently been
solved by Dekking, Shallit, and Sloane~\cite{DSS}.
They study the Greedy Queen in a spiral game, which is different from Splythoff Nim, but
has the same $P$-positions.

\section{A mex rule for the $P$-positions}

We will recursively generate the columns in the table of $P$-positions
of Splythoff Nim. We introduce some notation.
The minimal excluded value $\mathrm{mex}(S)$ of a proper subset $S\subset \mathbb N$
is the least element of the complement $\mathbb N\setminus S$.
Furthermore, $S_i$ denotes the subset of the least $i$ elements of $S$ and
$s_i$ denotes the $i$-th element of $S$.
In particular, $S_0$ is the empty set and $S_i=\{s_1,\ldots,s_i\}$.
In this notation,
Wythoff's first method to generate his table is $a_{i+1}=\mathrm{mex}(A_i\cup B_i),\
b_{i+1}=a_i+i+1$.
We will show below that the columns
of Table~\ref{tbl4} are generated by
\begin{equation}\label{mexrule}
\begin{array}{rl}
\delta_{i+1}=&\mathrm{mex}\left(\Delta_i\cup\Sigma_i\right)\\
a_{i+1}=&\mathrm{mex}\left(A_i\cup B_i\right)\\
b_{i+1}=&a_{i+1}+\delta_{i+1}\\ \sigma_{i+1}=&a_{i+1}+b_{i+1}.
\end{array}
\end{equation}
All four sequences are strictly increasing.
Since $A$ and $\Delta$ are defined by the $\mathrm{mex}$, both $\{\Delta, \Sigma\}$
and $\{A, B\}$ partition $\mathbb N$.
\medbreak

Splythoff Nim is a two-pile game with positions $(a,b)$ in which
$a$ is the number of counters on the smallest pile.
We write $\delta(a,b)=b-a$ and $\sigma(a,b)=a+b$.
Suppose there is a move from $(a,b)$ to $(a',b')$.
If the move is a single, then either $a$ or $b$ is untouched
and $\{a,b\}\cap\{a',b'\}$ is non-empty. If the move is a double,
then $\delta(a,b)=\delta(a',b')$. If the move is a split,
then $\delta(a,b)=\sigma(a',b')$.
We summarize this as a lemma.

\begin{lemma}\label{lem1}
If there is a move from $(a,b)$ to $(a',b')$ then
$\{a,b\}\cap\{a',b'\}$ or $\{\delta(a,b),\sigma(a,b)\}\cap
\{\delta(a',b'),\sigma(a',b')\}$ is non-empty.
\end{lemma}

Observe that $\{a,b\}$ is a multiset if $a=b$ and that
$\{\delta(a,b),\sigma(a,b)\}$ is a multiset if $a=0$.
These are the positions with a move to $(0,0)$.

Let $\mathcal P$ denote the set of $P$-positions
in an arbitrary impartial game.
The standard method to construct $\mathcal P$
is by recursion.
Let $\mathcal P_0$ be the positions with no move.
Let $\mathcal M_0$ be the positions with a move to $\mathcal P_0$.
Remove $\mathcal M_0$ from the game.
Let $\mathcal P_1$ be the positions with no move in this reduced game.
Let $\mathcal M_1$ be the positions with a move to $\mathcal P_0\cup\mathcal P_1$.
Remove $\mathcal M_1$ from the game, etc.
The union of all $\mathcal P_i$ is equal to $\mathcal P$.

In Splythoff Nim, $\mathcal P_0=\{(0,0)\}$ and $\mathcal M_0$
contains exactly those positions in which $\{a,b\}$ or $\{\delta,\sigma\}$
is a multi-set.
If we remove $\mathcal M_0$ from the game, then positions $(m,n)$
remain such that $0<m<n$.
In the reduced game, there is no move from $(1,2)$ but there is a move
from each other position. Therefore $\mathcal P_1=\{(1,2)\}$, which is the
first non-zero $P$-position and also equal to $(a_1,b_1)$ as generated
by the mex-rule of equation~\ref{mexrule}.

\begin{lemma}\label{lem2}
Let $a_i$ and $b_i$ be the $i$-th element generated by the mex-rule for $A$ and $B$.
Then $\mathcal P_i=\{(a_i,b_i)\}$.
\end{lemma}
\begin{proof}
Assuming that the statement is true for all $\mathcal P_i$ up to $\mathcal P_k$ for a fixed $k\geq 1$. We
prove that it is true for $\mathcal P_{k+1}$.
Let $\mathcal M_k$ be the set of all $(m,n)$ with a move to $\mathcal P_i$ for some $i\leq k$.
We remove these positions from the game. The positions with no move in the reduced
game are in $\mathcal P_{k+1}$.
CLAIM: $\mathcal M_k$ is equal to the set of positions $(m,n)$ such that one of the following holds:
\begin{itemize}
\item[(i)] $\{m,n\}$ intersects $A_k\cup B_k$,
\item[(ii)] $\{\delta(m,n),\sigma(m,n)\}$ intersects $\Delta_k\cup\Sigma_k$.
\end{itemize}
Lemma~\ref{lem1} implies that if $(m,n)\in\mathcal M_k$, then (i) or (ii)
holds. We need to prove the converse.

If (i) holds, then one of the coordinates of $(m,n)$ is equal to a coordinate
in $(a_i,b_i)$ for some $i\leq k$.
If the other coordinate of $(m,n)$ is larger, then there is a single to $(a_i,b_i)$
and we are done. We may assume that the other coordinate is smaller.
If $\delta(m,n)<\delta(a_i,b_i)$ then there exists a $j<i$ such that
$\delta(m,n)=\delta(a_j,b_j)$ or $\delta(m,n)=\sigma(a_j,b_j)$ because
$\delta(a_i,b_i)$ is a mex.
In the first case, there is a double to $(a_j,b_j)$ and in the second
case there is a split.
So we are left with the case that $\delta(m,n)\geq\delta(a_i,b_i)$ and
the other coordinate of $(m,n)$ is smaller, which implies that
the equal coordinate is $n$ and $m<a_i$. Since $a_i$ is a mex, $m=a_j$
or $m=b_j$ for some $j<i$. Since $\delta(m,n)\geq \delta(a_i,b_i)>\delta(a_j,b_j)$,
there is a single from $(m,n)$ to $(a_j,b_j)$. Hence if (i) holds then
$(m,n)\in\mathcal M_k$.

Now consider case (ii). We may assume that (i) does not hold and therefore
$m>a_k$ by the mex-rule.
If $\delta(m,n)\in\{\delta_i,\sigma_i\}$ for some $i\leq k$ there is either
a double or a split from $(m,n)$ to $(a_i,b_i)$.
If $\sigma(m,n)\in\{\delta_i,\sigma_i\}$ then
we must have $\sigma(m,n)\leq\sigma_i$ which implies $b_i>n$ because $m>a_i$.
It follows that $\delta(m,n)<\delta_i$ and there is a double from $(m,n)$
to $(a_j,b_j)$ for some $i<j$.
All positions that satisfy (i) or (ii) have a move to a $P_j$ for some $j\leq k$.
The claim holds.

To finish the proof, we need to show that $(a_{k+1},b_{k+1})$ is the unique position that has no move
if we reduce the game by $\mathcal M_k$.
It follows from the mex-rule in equation~\ref{mexrule}
that $a_{k+1}$ and $\delta_{k+1}$ are not contained in $A_k\cup B_k$ or
$\Delta_k\cup \Sigma_k$, respectively. Since the sequences are increasing,
neither are $b_{k+1}$ and $\sigma_{k+1}$. Therefore, $(a_{k+1},b_{k+1}$
does not satisfy (i) or (ii) and therefore it is in the complement
of $\mathcal M_k$.
All moves from $(a_{k+1},b_{k+1})$ end up in $\mathcal M_k$.
Indeed, if a move reduces $a_{k+1}$ then it is in $A_k\cup B_k$ by the mex-rule.
Similarly, if it reduces $\delta_{k+1}$ then it is in $\Delta_k\cup\Sigma_k$.
If a move fixes both $a_{k+1}$ and $\delta_{k+1}$, then $b_{k+1}$ is reduced
to a number smaller than $a_{k+1}$, which is in $A_k\cup B_k$. Whatever
the move, the resulting position is in $\mathcal M_k$.
Therefore, $(a_{k+1},b_{k+1})\in\mathcal P_{k+1}$.

Let $(m,n)$ be any position in the complement of $M_k$
unequal to $(a_{k+1},b_{k+1})$.
We show that there exists a move from $(m,n)$
within the complement. We have that $m\geq a_{k+1}$ since $m\not\in A_k\cup B_k$.
and similarly that $\delta(m,n)\geq \delta_{k+1}$. It follows that $n\geq b_{k+1}$.
If $m>a_{k+1}$ then $n>b_{k+1}$. If we remove one counter from both piles,
then the resulting position $(m-1,n-1)$ still is in the complement.
If $m=a_{k+1}$ then $n>b_{k+1}$ and there is a single to $(a_{k+1},b_{k+1})$.
We conclude that there is a move in the complement.
\end{proof}

We conclude that the mex-rule generates the columns in Table~\ref{tbl4} of
the $P$-positions
in Splythoff Nim. We now proceed with the proof of Theorem~\ref{thm1}, which will
follow as a by-product of a more general property of the $k$-bonacci substitution.
Like Wythoff, we will be interested in tables in their own right.

\section{The positions table of the $k$-bonacci substitution}

The $k$-bonacci substitution $\theta_k$ on the alphabet $\{0,\ldots,k-1\}$ is given by
\[
\theta_k\colon\left\{
\begin{array}{rcll}
j&\mapsto& 0\:(j+1),&\mathrm{if }\ j<k-1\\
k-1&\mapsto& 0
\end{array}.\right.
\]
The $k$-bonacci word $\omega^{k}$ is the
unique fixed point of this substitution.
In our analysis, $k>2$ is fixed and we will often simply write $\theta$
and $\omega$ instead of $\theta_k$ and~$\omega^k$.
We start by recalling some well-known
properties of the $k$-bonacci word, before turning to the tables.

\begin{lemma}\label{lem3}
The letters $j>0$ are isolated in $\omega$, i.e, each is preceded and followed by a $0$.
\end{lemma}
\begin{proof}
It follows from the definition of $\theta$ that each $j>0$ is preceded by a $0$.
It has to be succeeded by $0$ as well, for the same reason.
\end{proof}

In fact, we can say more.	

\begin{lemma}\label{lem6}
In $\omega$ each $j$ occurs as the final letter of $\theta^j(0)$.
\end{lemma}
\begin{proof}
Each $j$ occurs as the final letter of $0j=\theta(j-1)$, which by
induction is the suffix of $\theta(\theta^{j-1}(0))$.
\end{proof}

The length of $\theta^j(0)$ is $2^j$. We denote the prefix of $j$
by $w_{j}$, i.e., $\theta^j(0)=w_{j}j$.
If $j=0$ then it is the empty word $w_{0}=\epsilon$.

\begin{lemma}\label{lem4b}
In $\omega$ each $j$ is preceded and followed by $w_{j}$.
\end{lemma}
\begin{proof}
The case $j=0$ is trivial.
Each $j > 0$ can only occur as the final letter of $\theta(j-1)$. Assume for induction that
$j-1$ is preceded and followed by $w_{j-1}$. Therefore, $j$ can only occur in
\[
\theta(w_{j-1})0j\theta(w_{j-1}).
\]
The final letter of $w_{j-1}$ is $0$ and therefore the final letter of
$\theta(w_{j-1})$ is $1$, which is followed by a $0$. Therefore, each
$j$ can only occur in
\[
\theta(w_{j-1})0j\theta(w_{j-1})0.
\]
We see that
\[
w_jj = \theta^j(0) = \theta(\theta^{j-1}(0)) = \theta(w_{j-1}(j-1)) = \theta(w_{j-1})0j
\]
and so $\theta(w_{j-1})0 = w_j$. Each $j$ is preceded and followed by $w_j$.\end{proof}

Each $w_j$ can be recursively defined by
\[
w_{j+1}=w_{j}jw_{j}.
\]
This follows from \[w_{j+1}=\theta(w_{j})0=
\theta(w_{j-1}(j-1)w_{j-1})0=
\theta(w_{j-1})0j\theta(w_{j-1})0=w_jjw_j.\]
In particular, $w_j$ is a palindrome. We saw that each $j$ is preceded \emph{and} followed
by $w_j$. Conversely, each $w_j$ is preceded \emph{or} followed by $j$.

\begin{lemma}\label{lem4c}
In $\omega$ each $w_j$ is a prefix or a suffix of a $j$.
\end{lemma}
\begin{proof}
By induction. We have that $w_{j}=\theta(w_{j-1})0$. If $w_{j}$ is followed by $j$ we are done.
Say it is followed by an $i$ unequal to $j$.
Then $w_ji=\theta(w_{j-1})0i=\theta(w_{j-1}(i-1))$. Since $w_{j-1}$ is followed by $i-1$
it must be preceded by $j-1$, which implies that $w_j$ is preceded by $j$.
\end{proof}

Note that it can happen that $w_j$ is both a prefix and a suffix. Since
$00$ occurs in $\omega$, so does $0101$ and here $w_1=0$ occurs as a prefix and
a suffix of~$1$.

Lemma~\ref{lem6} implies that the $n$-th occurrence
of $j$ in $\omega=\omega_1\omega_2\omega_3\cdots$
is in the $n$-th word of
$\theta^{j+1}(\omega_1)\theta^{j+1}(\omega_2)\theta^{j+1}(\omega_3)\cdots$.
For a letter $j$ we will say that the distance between consecutive
positions of $j$ are \textit{steps}.
The steps between $j$'s are equal to the lengths of the
$\theta^{j+1}(\omega_i)$. This is illustrated in the positions
table for the Quadribonacci word below.

\begin{table}[ht]
{\fontsize{8}{10}\selectfont
\begin{tabular}{c|ccccccccccccccccc}
\rowcolor[gray]{.8}$\omega^4$
     &0&1 &0 &2 &0 &1 &0 &3  &0  &1  &0  &2  &0  &1  &0  &0&\\
\hline
$X^0$&1&3 &5 &7 &9 &11&13&15 &16 &18 &20 &22 &24 &26 &28 &30\\
$X^1$&2&6 &10&14&17&21&25&29 &31 &35 &39 &43 &46 &50 &54 &58\\
$X^2$&4&12&19&27&33&41&48&56 &60 &68 &75 &83 &89 &97 &104&112\\
$X^3$&8&23&37&52&64&79&93&108&116&131&145&160&172&187&201&216
\end{tabular}}
\\[9pt]
	\caption{\small The positions table of the Quadribonacci word.
The steps between columns are equal to
$(2,4,8,15),(2,4,7,14),(2,3,6,12),(1,2,4,8)$, depending
on the letter in the top row. If we add or subtract $4$
in the fourth row, then we get entries in the third row, etc. Equivalent properties
hold for all $k$-bonacci tables.
}
	\label{tbl2a}
\end{table}

The difference between the fourth and the
fifth column (and the twelfth and the thirteenth)
is $(2,3,6,12)$ because $2$ is the fourth letter (and twelfth)
in~$\omega^4$. For the positions table of the general $k$-bonacci word,
the result is as follows.

\begin{lemma}\label{lem7}
Let $\ell^j(i)$ be the length of $\theta^{j+1}(i)$. The positions table
can be generated from the vector valued substitution
\begin{equation}\label{eq3}
\nu\colon
i\mapsto
\left[
\begin{array}{c}
\ell^0(i)\\
\ell^1(i)\\
\vdots\\
\ell^{k-1}(i)
\end{array}
\right].
\end{equation}
The initial column of the table contains increasing powers of $2$, starting from $2^0$,
and the $n+1$-th column is generated from the $n$-th column by adding
$\nu(\omega^k_n)$.
\end{lemma}
\begin{proof}
By Lemma~\ref{lem6} the letter $j$ first occurs in $\theta^{j+1}(\omega_0^k)$ at position $2^j$.
The first column of the table therefore contains the powers of two. Since $j$ occurs
at the same location in each $\theta^{j+1}(i)$, independent of $i$, each row $X^j$ has step
sizes $\ell^j(i)$.
\end{proof}
The length $\ell^j(i)$ is minimal if $i=k-1$, in which case the length is equal to $2^{j}$.
The other lengths are greater than $2^j$.
Therefore $2^j$ is the minimal step in the $j$-th row $X^j$
and it occurs for columns headed by
$k\!-\!1$.

\begin{lemma}\label{lem8}
Both $X^{j+1}-2^{j}$ and $X^{j+1}+2^j$ are subsets of $X^j$
and their union is equal to $X^j$.
An element of $X^j$ can be written as $m+2^j$ and $n-2^j$ exactly
if $m$ is taken from a $k\!-\!1$-column in the positions table, and
$n$ is its successor.
\end{lemma}
\begin{proof}
Suppose that $j+1$ occurs at location $m\in X^{j+1}$.
By Lemma~\ref{lem6} we have that both $m-2^j$ and $m+2^j$ are elements of $X^j$.
Hence $X^{j+1}-2^{j}$ and $X^{j+1}+2^j$ are subsets of $X^j$.
Conversely, suppose that $h\in X^j$, i.e., $j$ occurs at location $h$ in $\omega$.
Each $j$ occurs in a $w_j$, which is either a prefix or a suffix of a $w_{j+1}$.
Within $w_j$, $j$ occurs in the middle location $2^j$.
Since $w_j$ is a palindrome, the distance between $j$ and $j+1$
in $w_{j+1}=w_jj+1w_j$ is the same for the prefix and the suffix.
That distance is equal to $2^j$.
Therefore, $h\in X^{j+1}\pm 2^j$.
A number in $X^j$ can be written both as a sum and a difference if it occurs in a
step of size $2^{j+1}$ in $X^{j+1}$. These steps occur at the $k\!-\!1$-columns.
\end{proof}

\section{The difference table}

The rows in the positions table form partition $\mathbb N$, by definition.
Remarkably, this also holds for the rows of the {\em difference table},
which has rows $\Delta^j=X^{j+1}-X^j$. We will see that these
rows $\Delta^j$ again form a partition of~$\mathbb N$.
\begin{table}[ht]\label{tbl6}
{\small\begin{tabular}{c|ccccccccccccccccc}
\rowcolor[gray]{.8}$\omega^4$
     &0&1 &0 &2 &0 &1 &0 &3  &0  &1  &0  &2  &0  &1  &0  &0&\\
\hline
$\Delta^0$&1&3 &5 &7 &8 &10&12&14 &15 &17 &19 &21 &22 &24 &26 &28\\
$\Delta^1$&2&6 &9&13&16&20&23&27 &29 &33 &36 &40 &43 &47 &50 &54\\
$\Delta^2$&4&11&18&25&31&38&45&52 &56 &63 &70 &77 &83 &90 &97&104
\end{tabular}}
\\[9pt]
	\caption{\small The difference table for the Quadribonacci word.
Steps between columns depend
on the letter in the header. If we add or subtract $2$
in the third row, then we get entries in the second row, etc.
We prove below that
the difference table of the Tribonacci word produces the $P$-positions
in Table~\ref{tbl4}.
}
\end{table}
\newline It follows from Lemma~\ref{lem7} and Equation~\ref{eq3}
that the difference table can be derived from $\omega$. The step
size is
\begin{equation}\label{eq4}
\left[
\begin{array}{c}
\ell^1(i)-\ell^0(i)\\
\ell^2(i)-\ell^1(i)\\
\vdots\\
\ell^{k-1}(i)-\ell^{k-2}(i)	
\end{array}
\right]
\end{equation}
if $i$ is the letter heading the column.

\begin{lemma}\label{lem9}
For each $1\leq h\leq k-1$ and each letter $i$
\[\ell^{h}(i)=2\ell^{h-1}(i)-\delta^{k-1}_{h+i},\]
where $\delta^i_j$ is Kronecker's delta.	
\end{lemma}
\begin{proof}
We defined $\ell^h(i)$ as the length of $\theta^{h+1}(i)$.
The substitution replaces each letter by two letters, unless it is the final letter $k-1$,
which is replaced by a single $0$.
The lemma thus
says that $\theta^h(i)$ contains no letter $k-1$ unless $k-1=h+i$, in which case it
contains one letter $k-1$.

If we start from $i$, the only letter that reaches $k-1$
is produced from $i\to i+1\to\cdots$. All other letters are produced from a $0$
and do not reach $k-1$ because $h\leq k-1$.
The sequence $i\mapsto i+1$ reaches $k_1$ after $k-1-i$ substitutions.
Hence, we have one non-doubling letter if $h=k-1-i$.
\end{proof}

We already observed that $\ell^h(i)>\ell^h(k-1)$ if $i<k-1$. Therefore
$\ell^h(i)-\ell^{h-1}(i)\geq
\ell^{h-1}(i)-\delta^{k-1}_{h+i}\geq
\ell^{h-1}(k-1)=2^{h-1}$. It follows that
the minimal step in $\Delta^{h-1}$ is equal to $2^{h-1}$.
Again, the minimal step occurs at the columns
headed by $k-1$. However,
the minimal steps do not exclusively occur in these columns.
For instance, in Table~\ref{tbl6} the minimal step in the
first row is $1$. It occurs in columns headed by $2$ and $3$.

\begin{lemma}\label{lem10}
$\Delta^{j+1}-2^{j}$ and $\Delta^{j+1}+2^j$ are subsets of $\Delta^j$
and their union is equal to $\Delta^j$.
An element of $\Delta^j$ can be written as $m+2^j$ and $n-2^j$ exactly
if $m$ is taken from a $k\!-\!1$-column in the positions table, and
$n$ is its successor.
\end{lemma}
\begin{proof}
By Lemma~\ref{lem8} we have
\[
\Delta^j=X^{j+1}-X^j=
\left(X^{j+2}-2^{j+1}\cup X^{j+2}+2^{j+1}\right)
-
\left(X^{j+1}-2^{j}\cup X^{j+1}+2^{j}\right).
\]
The signs at $2^{j+1}$ and $2^j$ depend on the heading letter and therefore
\[
\Delta^j=\left(X^{j+2}-X^{j+1}-2^{j+1}+2^j\right)\cup \left(X^{j+2}-X^{j+1}+2^{j+1}-2^j\right).
\]
which is $\Delta^{j+1}-2^j \cup \Delta^{j+1}+2^j$.
\end{proof}

\begin{lemma}\label{lem11}
The rows in the difference table are disjoint as sets:
\[\Delta^i\cap \Delta^j=\emptyset \ \text{if }i\not=j.\]
\end{lemma}
\begin{proof}
Suppose $j>i$. By iterating Lemma~\ref{lem10} we find that each
element of $\Delta^i$ is in some
$\Delta^j\pm 2^{j-1}\pm \cdots\pm 2^{i}$. The minimal step
size in $\Delta^j$ is $2^j$ and $2^{j-1}+\cdots+1<2^j$.
Therefore, $\Delta^i$ and $\Delta^j$ are disjoint.
\end{proof}

\begin{theorem}
The rows in the difference table form a partition of $\mathbb N$.
\end{theorem}

\begin{proof}
We only need to prove that the rows cover $\mathbb N$.
By iterating Lemma~\ref{lem10} each $\Delta^i$ is the union
of all $\Delta^{k-2}\pm 2^{k-2}\pm \cdots\pm 2^{i}$,
with $\Delta^{k-2}$ the bottom row of our table.
The union of the rows is equal to the union
of $\Delta^{k-2}+n$ for all integers $n$ that can be written as
$\pm 2^{k-2}\pm \cdots\pm 2^{i}$ for some $i$
and some choice of the signs. Here we take $n=0$ if $i=k-2$.
It is not hard to verify that each $n\in\{-2^{k-2}+1,\ldots, 2^{k-2}-1\}$
admits such an expansion. Hence, it suffices to prove that
the maximal step in $\Delta^{k-2}$ is $2^{k-1}-1$.

By Lemma~\ref{lem9} and Equation~\ref{eq4} the steps in $\Delta^{k-2}$ are given
by \[\ell^{k-1}(i)-\ell^{k-2}(i)=\ell^{k-2}(i)-\delta^i_0.\]
This is maximal if $i=0$ or $i=1$, when it is indeed equal to $2^{k-1}-1$.
\end{proof}

The difference table of the Tribonacci word $\omega^3$ consists of two rows.
We define the word $\nu=\nu_1\nu_2\nu_3\cdots$ by taking $\nu_j=0$
if $j$ is in the first row and $\nu_j=1$ if it is in the second row.
We will prove below that the difference table of $\omega^3$ gives the
$P$-positions of Splythoff Nim and so $\nu$ is the word given by
Equation~\ref{ABword}. We are now able to prove that it can be
coded by~$\omega^3$ as already described by Equation~\ref{codedword}.

\begin{theorem}\label{codingnu}
If we delete every $2$ from $\omega^3$ then we get $\nu$ corresponding
to the $P$-positions of Splythoff Nim.
\end{theorem}
\begin{proof}
According to our main result, proved below, shows that the difference
table of $\omega^3$ corresponds to the $P$-positions.
The step size in the first row of this difference table
is two in columns headed by $0$ and it is one in the columns headed
by $1$ or $2$. This implies that the first row corresponds to the locations
of $0$ if we apply the coding $0\mapsto 01,\ 1\mapsto 0,\ 2\mapsto 0$.
We denote this coding by $\kappa$. The Tribonacci word is the limit of
$\theta^n(0$. We conclude that $\nu$ is the limit of $\kappa(\theta^n(0))$.

Let $\lambda$ be the coding $0\mapsto 0,\ 1\mapsto 1,\ 2\mapsto\epsilon$,
where as before $\epsilon$ denotes the empty word. The coding $\lambda\circ\theta$
is equal to $\kappa$. Therefore $\lambda(\theta^{n+1}(0))=\kappa(\theta^n(0))$
and taking limits gives $\lambda(\omega^3)=\nu$.
\end{proof}

Our proof of Theorem~\ref{thm1}depends on an analysis of the double difference table,
which applies to all $k$-bonacci words.

\section{The double-difference table}

We continue by differencing the differences
\[
d\Delta^{i}=\Delta^{i+1}-\Delta^i
\]
and collect these as rows in a {\em double-difference table}.
It turns out that we need to add a bottom row containing the sum
\[
\Sigma=\Delta^0+\cdots+\Delta^{k-2}
\]
to turn the rows of the table into a partition.
We call this the {\em sum row} and we call the other rows the {\em difference rows}.
Again, we add the $k$-bonacci word $\omega$ in the header.
For the Tribonacci word we have one difference
row and one sum row. These are the $\Delta$
and $\Sigma$ of Table~\ref{tbl4} and that is how the double-difference table
will appear in our proof of Theorem~\ref{thm1}.

\begin{table}[ht]
{\small\begin{tabular}{c|ccccccccccccccccc}
\rowcolor[gray]{0.8}$\omega^4$&0&1 &0 &2 &0 &1 &0 &3  &0  &1  &0  &2  &0  &1  &0  &0\\
\hline
$d\Delta^{0}$&1&3 &4&6 &8 &10&11&13 &14 &16 &17 &19 &21 &23 &24 &26\\
$d\Delta^{1}$&2&5 &9&12&15&18&22&25 &27 &30 &34 &37 &40 &43 &47 &50\\
\rowcolor[gray]{0.95}$\Sigma$       &7&20&32&45&55&68&80&93 &100 &113 &125 &138 &148 &161 &173&193
\end{tabular}}
\\[9pt]
	\caption{\small The double-difference table for our standing example, the Quadribonacci word. If we add or subtract 2
from entries in the third row, then we get entries in the second row. However, we dot not get every
entry.}
\end{table}
Lemma~\ref{lem9} implies that the step between columns is
\begin{equation}\label{eq5}
\left[
\begin{array}{c}
\ell^2(i)-2\ell^1(i)+\ell^0(i)\\
\vdots\\
\ell^{k-1}(i)-2\ell^{k-2}(i)+\ell^{k-3}(i)\\
\ell^{k-1}(i)-\ell^0(i)
\end{array}
\right]
=
\left[
\begin{array}{rcl}
\ell^0(i)&-&\delta^{k-1}_{2+i}\\
&\vdots&\\
\ell^{k-3}(i)&-&\delta^{k-1}_{k-1+i}\\
\ell^{k-1}(i)&-&\ell^0(i)
\end{array}
\right].
\end{equation}
The minimum step occurs at $(k\!-\!1)$-columns, where
$\ell^h(i)-\delta^{k-1}_{h+2+i}=\ell^h(k-1)=2^h$
and $\ell^{k-1}(k-1)-\ell^0(k-1)=2^{k-1}-1$.
The maximum step in the final difference row is $\ell^{k-3}(i)-\delta^{k-1}_{k-1+i}$.
If $i\in\{0,1\}$ then the length of $\theta^h(i)$ is doubled at each
step up to $k-2$. If $i>1$ then at some step it is doubled minus one (see Lemma~\ref{lem9}).
Therefore $\ell^{k-3}(i)$ has maximal value
$2^{k-3}$ at $i=0,1$. Now $\delta^{k-1}_{k-1+i}=1$ if $i=0$ and it
is $1$ at $i=1$. Therefore, the step $\ell^{k-3}(i)-\delta^{k-1}_{k-1+i}$
is maximal in the columns marked by $1$.
Let $M\subset d\Delta^{k-3}$ be the
subsequence that is taken from these $1$-columns.

The steps in $M$ are sums of steps in $\Delta^{k-3}$ along
the intermediate letters between two $1$'s. In other words,
we the $j$-th step in $M$ is a sum over the $j$-th {\em return word}~\cite{Dur98}:
the $j$-th
word in $\omega$ that starts with a $1$ and leads up to, but
does not include, the $j+1$-th $1$.
The $1$'s occur as the second letters in the sequence
$\omega=\theta^2(\omega_1)\theta^2(\omega_2)\theta^2(\omega_3)\cdots$.
If $\omega_j=i < k-2$ then the $j$-th return word is $10(i+2)0$.
If $\omega_j=k-2$ then the return word is $100$
and if $\omega_j=k-1$ then it is $10$.

\begin{lemma}\label{lem12}
The $j$-th step in $M$ has length $\ell^{k-1}(\omega_j)-\ell^{0}(\omega_j)$.
\end{lemma}
\begin{proof}
If the $j$-th return word is $10(i+2)0$
for $i<k-2$, then the sum of the steps
is
\[
\ell^{k-3}(1)+\ell^{k-3}(0)-1+\ell^{k-3}(i+1)+\ell^{k-3}(0)-1.
\]
If we agree that $\ell(w)$ is the sum of $\ell$ over the letters in $w$,
and if we denote the $j$-th return word by $v_j$, then we can write this as
\[
\ell^{k-3}(v_j)-2=\ell^{k-3}(v_j)-\ell^0(i).
\]
This equation also holds for the other two return words $100$ if $i=k-2$
and $10$ if $i=k-1$.
By definition $\ell^{k-3}(w)$ is the length of $\theta^{k-2}(w)$
and $v_j=\theta^2(i)=\theta^2(\omega_j)$. We find that the sum of
the steps is
$
\ell^{k-3}(\theta^2(\omega_j))-\ell^0(\omega_j).
$
\end{proof}

We now repeat Lemmas~\ref{lem10} and~\ref{lem11} for the difference
rows in the double-difference table.

\begin{lemma}\label{lem13}
$d\Delta^{j+1}-2^{j}$ and $d\Delta^{j+1}+2^j$ are subsets of $d\Delta^j$
and their union is equal to $d\Delta^j$.
An element of $d\Delta^j$ can be written as $m+2^j$ and $n-2^j$ exactly
if $m$ is taken from a $k\!-\!1$-column in the positions table.
\end{lemma}
\begin{proof}
The same as the proof of Lemma~\ref{lem10}, with minor editing:
\[
d\Delta^j=\Delta^{j+1}-\Delta^j=
\left(\Delta^{j+2}-2^{j+1}\cup \Delta^{j+2}+2^{j+1}\right)
-
\left(\Delta^{j+1}-2^{j}\cup \Delta^{j+1}+2^{j}\right).
\]
Again, the signs depend on locations and so this is equal to
\[
\left(\Delta^{j+2}-\Delta^j-2^{j+1}+2^j\right)\cup \left(\Delta^{j+2}-\Delta^{j+1}+2^{j+1}-2^j\right)
\]
which is $d\Delta^{j+1}-2^j \cup d\Delta^{j+1}+2^j$.
\end{proof}

\begin{lemma}\label{lem14}
The difference rows in the double-difference table are disjoint as sets
\[d\Delta^i\cap d\Delta^j=\emptyset \ \text{if }i\not=j.\]
\end{lemma}
\begin{proof}
Suppose $j>i$. By iterating Lemma~\ref{lem10} we find that each
element of $\Delta^i$ is in some
$d\Delta^j\pm 2^{j-1}\pm \cdots\pm 2^{i}$. The minimal step
size in $d\Delta^j$ is $2^j$ and $2^{j-1}+\cdots+1<2^j$.
Therefore, $d\Delta^i$ and $d\Delta^j$ are disjoint.
\end{proof}

\begin{lemma}\label{lem15}
Let $S\subset\mathbb N$ be the set $S=M+2^{k-3}$.
The difference rows in the double-difference table form a partition of $\mathbb N\setminus S$.
\end{lemma}
\begin{proof}
We only need to prove that the difference rows cover $\mathbb N\setminus S$.
By iterating Lemma~\ref{lem13} each $d\Delta^i$ is the union
of all $d\Delta^{k-3}\pm 2^{k-4}\pm \cdots\pm 2^{i}$,
with $d\Delta^{k-3}$ the final difference row of our table.
The union of these rows is equal to the union
of $d\Delta^{k-3}+n$ for all integers $n$ that can be written as
$\pm 2^{k-3}\pm \cdots\pm 2^{i}$ for some $i$
and some choice of the signs (including $n=0$). As before we have
$n\in\{-2^{k-2}+1,\ldots, 2^{k-2}-1\}$.
The maximal step in $d\Delta^{k-3}$ is $2^{k-2}$, which
occurs in the $1$-columns. All other elements are covered.
\end{proof}

\begin{theorem}\label{thm3}
The rows in the double-difference table form a partition of $\mathbb N$.
\end{theorem}
\begin{proof}
By Lemma~\ref{lem12} the sequences $S$ and $\Sigma$ have equal steps.
We only need to show that they have the same initial element.
The initial element of $\Sigma$ is
the sum of the initial elements of the difference sequences. This
is $\ell^{k-1}(0)-\ell^0(0)=2^{k-1}-1$.
The initial element of $M$ is the second element of $d\Delta^{k-3}$.
The first element is $2^{k-3}$ and the second element is
$2^{k-3}+\ell^{k-3}(0)-1=2^{k-3}+2^{k-2}-1$. Finally,
the initial element of $S$ is $2^{k+3}+2^{k-3}+2^{k-2}-1=2^{k-1}-1$.
Indeed, $S$ and $\Sigma$ have the same initial element.
\end{proof}

\section{Proof of Theorem~\ref{thm1}}

We show how to generate the difference table of the Tribonacci word by
a mex-rule. This rule turns out to be the same as the rule for the $P$-positions
of Splythoff Nim, which settles the proof of Theorem~\ref{thm1}.

We denote the elements of the difference sequence $\Delta^j$ by $a^j_1,a^j_2,\ldots$
and the elements of $d\Delta^j$ by $d^j_1,d^j_2,\ldots$.
As before, $X_i$ denotes the first $i$ entries of a sequence $X$.
\begin{lemma}\label{lem16}
$a^0_{i+1}=\mathrm{mex}\left(\Delta^0_i\cup\Delta^1_i\cup\cdots\cup\Delta^{k-2}_i\right)$ .
\end{lemma}
\begin{proof}
The rows partition $\mathbb N$ and therefore
$\mathrm{mex}\left(\Delta^0_i\cup\Delta^1_i\cup\cdots\cup\Delta^{k-2}_i\right)$
occurs in one of the rows. Both rows and columns are strictly increasing.
This mex has to has to be in the first row.
\end{proof}

By the same argument we find
\begin{lemma}\label{lem17}
$d^0_{i+1}=\mathrm{mex}\left(d\Delta^0_i\cup d\Delta^1_i\cup\cdots\cup d\Delta^{k-3}_i\cup\Sigma_i\right)$.
\end{lemma}

We can now complete the final lemma to our proof of Theorem~\ref{thm1}.
\begin{lemma}\label{cor1}
The  difference table of the Tribonacci word
is identical to the rows of Table~\ref{tbl4}. The double difference table
is identical to the header and footer of that table.
\end{lemma}
\begin{proof}
If $k=3$ there are only the rows $\Delta_0, \Delta_1, d\Delta_0, \Sigma$.
The previous two lemmas state that
\begin{equation}
\begin{array}{rl}
d^0_{i+1}=&\mathrm{mex}\left(d\Delta^0_i\cup\Sigma_i\right)\\
a^0_{i+1}=&\mathrm{mex}\left(\Delta_i^0\cup\Delta_i^1\right)
\end{array}.
\end{equation}
By definition $a^1_{i+1}=a^0_{i+1}+d^0_{i+1}$ and if $s_i$  denotes the entries in $\Sigma$
then by defition $s_{i+1}=a^0_{i+1}+a^1_{i+1}$. These rules are the same as for the table
of $P$-positions in Equation~\ref{mexrule}.
 \end{proof}

\section{A $\mathrm{mex}$-rule for the Quadribonacci table}

The original goal of our paper was to find a simple impartial game that can be
coded by the Quadribonacci word.
We were unable to find such a game. The best that we can come up with
is a $\mathrm{mex}$-rule to generate the positions table of the Quadribonacci word.

\begin{lemma}\label{lem18}
The bottom row of the $k$-bonacci positions table can be written as a sum
\[
X^{k-1}=E+X^0+\cdots+X^{k-1}
\]
where $E$ denotes the enumerating sequence $1,2,3,\ldots$.
\end{lemma}
\begin{proof}
This is true for the initial column. Each next column is an increment by the
vector in Equation~\ref{eq3}. We need to show that
\[
\ell^{k-1}(i)=1+\ell^{0}(i)+\cdots+\ell^{j-2}(i).
\]
This is a consequence of the following observation. Suppose you
start from $1$
and double each time, except once, when you double and subtract one.
Then the final number is the sum of the other numbers.
We leave the verification to the reader. The equality above
follows from the observation for $i<k-1$. In that case,
we have $\ell^0(i)=2$ and each next $\ell^{j}(i)$ is doubled,
unless $k-1$ appears in $\theta^j(i)$, which happens once.
If $i=k-1$ then $\ell^{j}(k-1)=2^{j}$ and again the equation holds.
\end{proof}

According to this lemma it suffices to generate the first
three rows of the Quadribonacci table, and compute the fourth
row as a sum. These three rows can be derived from
the first two rows of the difference table, which can be derived
from the first row of the double-difference table, which follows
from the $\mathrm{mex}$ rule. That is the idea behind
Theorem~\ref{thm4} below, in which we generate all three tables of the
Quadribonacci word simultaneously.

We write $x^j_i$ for the elements of the positions table.
\begin{theorem}\label{thm4}
The following rules generate the three tables for the Quadribonacci word:
\begin{eqnarray*}
a^0_{i+1}&=&\mathrm{mex}\left(\Delta^0_i\cup\Delta^1_i\cup\Delta^{2}_i\right)\\
b^0_{i+1}&=&\mathrm{mex}\left(d\Delta^0_i\cup d\Delta^1_i\cup \Sigma_i\right)\\
x_{i+1}^0&=&\mathrm{mex}\left(X^0_i\cup X^1_i\cup X^{2}_i\cup X^3_i\right)\\
a^1_{i+1}&=&a^0_{i+1}+b^0_{i+1}\\ x^1_{i+1}&=&x^0_{i+1}+a^0_{i+1}\\ x^2_{i+1}&=&x^1_{i+1}+a^1_{i+1}\\
x^3_{i+1}&=&x^0_{i+1}+x^1_{i+1}+x^2_{i+1}+i+1\\
a^2_{i+1}&=&x^3_{i+1}-x^2_{i+1}\\
b^2_{i+1}&=&a^1_{i+1}+a^2_{i+1}+a^3_{i+1}.
\end{eqnarray*}
\end{theorem}
\begin{proof}
The first and second equation follow from Lemma~\ref{lem16} and \ref{lem17}.
The third equation can be derived in an equivalent manner. The other equations
follow from Lemma~\ref{lem17} and the definition of the difference tables.
\end{proof}

This method to generate the Quadribonacci table is comparable to Duch\^{e}ne and Rigo's
method to generate the positions table of cubic substitutions~\cite[Prop. 4]{DR3}.
Duch\^{e}ne and Rigo were able to describe a set of forbidden (malicious) moves
from their $\mathrm{mex}$-rule, which involves the difference table.
We were unable to find forbidden moves for the Quadribonacci word.
Our rule involves the double difference table and therefore one has to dig deeper.

\section{Some further remarks on Splythoff Nim}

We did not consider the Sprague-Grundy values of Splythoff Nim.
For Wythoff Nim, these values remain an object of study. It is
known that every row, column and diagonal of $\mathbb N\times \mathbb N$
contains each Sprague-Grundy value once and no more. More specifically,
if we fix one pile and we let the other pile increase from zero
to infinity, then the Sprague-Grundy values are a permutation of $\mathbb N$
(if we include zero as a natural number). The same is true if we fix the
difference between the piles and let the smallest pile increase from
zero to infinity~\cite{BF}.

For Splythoff Nim, it appears that every row or column contains each Sprague-Grundy
value once and no more. We do not have a proof for that. Diagonals do not contain
each value. For instance, the diagonal $(n,n+3)$ does not contain Sprague-Grundy
value zero. Similarly, the diagonal $(n,n+4)$ does not contain value one. Wythoff
Nim can be played with a queen on a chessboard. Splythoff Nim is played with a queen
that can reflect against the boundary. To get all Sprague-Grundy values, one has
to include the reflection of the diagonal. Again, we have no proof for that. 

\begin{table}[ht]
{\small\begin{tabular}{|ccc|ccc|ccc|ccc|ccc|ccc|ccc|}
\hline
 17&\it 13&\it 18&\it 20&\it 12&  11&\it 15&\it 22& \it  14&\it  4&\it  1&\it   19&\it   16&\it   10&\it  7&  24&  25& 	21\\
16&17&15&\it 19&\it 20&\it 14&\it 21&\it 12&\it 22&\it 0&\it 5&\it 8& 6&\it 24&\it 10& 9&13&25\\
15&16&\it 14&18&\it 19&\it 17&\it 20&\it 21&\it 12&\it 2&\it 4&\it 22&\it 23&\it  7&\it  8&\it 11& 9&24\\
\hline
14&\it 15&\it 16&\it 17&\it 18&\it 13&\it 12&\it 19&\it 1&\it 3&\it 20&\it 21&\it 22&\it 23&\it 9&\it 8&\it 10&\it 7\\
13&14&12&11&\it 8&\it16&17&\it 0&\it 9& 5& 6&\it 18&\it 21&\it 19&\it 23&\it 7&\it 24&\it 10\\
12&\it 9&\it 13&\it 7&11&\it 15&\it 14&\it 2&18&\it 8&\it 19&\it 20&10&\it 21&\it 22&\it 23& 6&\it 16\\
\hline
11&\it 12&10&\it 13&\it 14&\it 9&\it 0&\it 16&17& 6&\it 7& 15&\it 20&\it 18&\it 21&\it 22&\it 8&\it 19\\
10&11& 9& 8&13&12&\it 2&15&16&17&14&\it 7&\it 19& 6&\it 20&\it 4&\it 5&\it 1\\
 9&10&11&12&\it 1& 7&13&14&15&16&17& 6&\it 8& 5&\it 3&\it 2&\it 0&\it 4\\
\hline
 8& 6& 7&10&\it 0& 2& 5& 3& 4&15&16&17&18&\it 9&\it 1&\it 12&\it 22&\it 14\\
 7& 8& 6& 9&\it 10& 1& 4& 5& 3&14&15&\it 16&\it 2&\it 0&\it 19&\it 21&\it 12&\it 22\\
 6& 7& 8& 1& 9&10& 3& 4& 5&13&\it 2&\it 0&\it 14&17&\it 12&\it 20&\it 21&\it 15\\
\hline
 5& 3& 4& 0& 6& 8&10& 1& 2& 7&\it 12&\it 9&\it 15&\it 16&\it 13&\it 17&14&11\\
 4& 5& 3& 2& 7& 6& 9&\it 10&\it 0&\it 1&13&\it 14&11&\it 8&\it 18&\it 19&\it 20&\it 12\\
 3& 4& 5& 6& 2& 0& 1& 9&10&12& 8&\it 13&\it 7&11&\it 17&18&\it 19&\it 20\\
\hline
 2& 0& 1& 5& 3& 4& 8& 6& 7&11& 9&10&\it 13&12&\it 16&\it 14&15&\it 18\\  
 1& 2& 0& 4& 5& 3& 7& 8& 6&10&11&\it 12&\it 9&14&\it 15& 16& 17 &\it 13\\
 0& 1& 2& 3& 4& 5& 6& 7& 8& 9&10&11&12&13&14&15&16&17\\
\hline
\end{tabular}}
\\[9pt]
	\caption{\small The Sprague-Grundy values of Splythoff Nim 
for piles $(m,n)$ containing up to sixteen counters. It should be
compared to the table of Wythoff Nim~\cite[p. 74]{BCG}, 
described as `chaotic'. A Sprague-Grundy value is printed in italics 
if it is different for Wythoff Nim. }\label{tblSG}
\end{table}

The $k$-bonacci substitution is connected to numeration systems.
Indeed, Sirvent introduced the substitution in~\cite{S}
to study $k$-bonacci numberation systems.
The relation between Wythoff's Nim and Fibonacci numeration
is well known and has been extended to other Wythoff-like games
by~Fraenkel~\cite{F, FY, GF}.
We did not consider the relation between the $P$-positions of Splythoff Nim
in Table~\ref{tbl4}
and the Tribonacci numeration system, but it is the central idea in the analysis
of Dekking, Shallit, and Sloane of the Greedy Queens on a Spiral.
They build on numeration results of Carlitz, et al.~\cite{CSH} to
describe the coordinates of the $P$-positions in terms of Tribonacci numeration.

Let $t^j_n$ be the $n$-th entry in the $j$-th row of the positions table
of the $k$-bonacci word.
It appears that $t^j_n$ is well approximated by $n\lambda_k^j$.
For the Fibonacci word the
approximation is as close as possible: $-1<t^j_n-n\lambda^j_2<0$.
For the other $k$-bonacci words the approximation is less close.
Dekking, Shallit and Sloane~\cite{DSS} give sharp
bounds on $t_3^j-n\lambda^j_3$.
The difference $t_n^j-n\lambda_k^j$ is known as the {\em symbolic discrepancy}.
Adamczewski~\cite{A} proved that all Pisot substitutions have
bounded discrepancy.   
The $k$-bonacci substitutions are Pisot and therefore
the difference $t_j^n-n\lambda_k^j$ in our table is bounded.

\section{$a$-Splythoff Nim}

Recall that $a$-Wythoff is a modification of Wythoff Nim in which a player may remove
$x$ coins from one pile and $y$ coins from the other, if $|x-y|<a$.
We can also modify $a$-Wythoff by allowing a split.
We say that a move is a \emph{double} if it removes coins from both piles.
If one of the piles is cleared by
a double, then in $a$-Splythoff the player may split the remaining pile.
It appears that $2$-Splythoff can also be coded by the Tribonacci word, as illustrated by
Table~\ref{tbl8}. The $B$-row appears to be equal to the difference between the
first row and the third row in the positions table of $\omega^3$.

\begin{table}[ht]
{\small\begin{tabular}{c|ccccccccccccccccc}
\rowcolor[gray]{0.8}$\omega^3$&0&1&0&2&0&1&0&0&1&0&2&0&1&0&1&0&2\\
\hline
$A$&1 &2 &4 &5 &6 &7 &9 &10 &11 &13 &14 &15 &16 &18 &19 &21&22\\
$B$&3 &8 &12 &17 &20 &25 &29 &34 &39 &43 &48 &51 &56 &60 &65 &69&74
\end{tabular}}
\\[9pt]
	\caption{\small The $P$-positions for $2$-Splythoff Nim. The steps appear
to be coded by the Tribonacci word: $0$ corresponds to a step of $(1,5)$,
$1$ corresponds to $(2,4)$, and $2$ corresponds to $(1,3)$.}\label{tbl8}
\end{table}

The next table contains the first $P$-positions of $3$-Splythoff Nim.
Again, there appear to be only three steps in the table, which we code by $0, 1, 2$.
The code word appears to be fixed under the substitution $0\mapsto 01$,
$1\mapsto 2$, $2\mapsto 01$.

\begin{table}[ht]
{\small\begin{tabular}{c|ccccccccccccccccc}
\rowcolor[gray]{0.8}
&  0&  1& 2& 0& 1&  0&  1& 2&  0&  1&2&  0&  1&  0&  1&2\\
\hline
$A$
&  1&  2&  3&  5&  6&  7&  8&  9& 11& 12& 13& 15& 16& 17& 18& 19\\
$B$ &  4& 10& 14& 20& 26& 30& 36& 40& 46& 52& 56& 62& 68& 72& 78& 82
\end{tabular}}
\\[9pt]
	\caption{\small The $P$-positions for $3$-Splythoff Nim. The header codes
the steps in the table.}
\end{table}

Finally, we look at $4$-Splythoff. Now the table appears to have four steps,
but the code of the steps is
$
012302010420121013002312011132\cdots
$
which does not seem to be fixed by a substitution on four letters.

\begin{table}[ht]
{\small\begin{tabular}{c|ccccccccccccccccc}
\rowcolor[gray]{0.8}
&  0&  1&  2&  3&  0&  2&  0&  1&  0&  4&  2&  0&  1&  2&  1&  0&  1\\
\hline
$A$&  1&  2&  3&  4&  6&  7&  8&  9& 10& 11& 13& 14& 15& 16& 17& 18& 19\\
$B$&  5& 12& 21& 26& 34& 41& 46& 53& 62& 69& 79& 84& 91&100&105&114&121
\end{tabular}}
\\[9pt]
	\caption{\small The $P$-positions for $4$-Splythoff Nim. The header codes
the steps in the table.}
\end{table}

We leave these observations as an open question:
\medbreak
\fbox{\begin{minipage}{30em}
\textbf{Question:} Is it true that the table of $P$-positions for
$a$-Splythoff Nim has a finite number of steps
for each $a$? Is the code word of the steps fixed under a substitution
on three letters if $a=2$ or $a=3$?\end{minipage}}

\section{Acknowledgement}
We thank Michel Dekking for bringing the Greedy Queens
in a spiral to our attention, and we thank Dirk Frettl\"oh and 
Jamie Walton for useful conversations.
Dan Rust would like to acknowledge
support of the Dutch Science Federation (NWO) through visitor
grant 040.11.700 and the German Research Foundation (DFG) via the
Collaborative Research Centre (CRC 1283).

\bigskip
\noindent
Institute of Applied Mathematics\\
Delft University of Technology\\
Mourikbroekmanweg 6 \\
2628 XE Delft, The Netherlands\\
\texttt{r.j.fokkink@tudelft.nl}
\\[4mm]
Fakult\"{a}t f\"{u}r
Mathematik\\
 Universit\"{a}t
Bielefeld\\
Universit\"{a}tsstrasse 25\\
D-33615 Bielefeld, Germany \\
\texttt{drust@math.uni-bielefeld.de}
\end{document}